\documentclass[10pt]{amsart}
\usepackage{amssymb,amsmath,amsthm}
\usepackage[colorlinks=true]{hyperref}
\usepackage{mathtools}
\usepackage{enumitem}

\usepackage{tikz}
\usetikzlibrary{arrows,decorations.pathmorphing,backgrounds,positioning,fit,petri}
\usetikzlibrary{decorations.markings}
\tikzset{->-/.style={decoration={
  markings,
  mark=at position .5 with {\arrow{>}}},postaction={decorate}}}

\DeclarePairedDelimiterX{\norm}[1]{\lVert}{\rVert}{#1}

\usepackage{cleveref}
\crefformat{section}{\S#2#1#3}

\theoremstyle{plain}
\newtheorem{theorem}{Theorem}[section]

\newtheorem{lemma}[theorem]{Lemma}
\newtheorem{cor}[theorem]{Corollary}
\newtheorem{prop}[theorem]{Proposition}
\theoremstyle{definition}

\begin{document}
\title{A conjecture on the lengths of filling pairs}

\author{Bidyut Sanki}
\address{Department of Mathematics and Statistics, Indian Institute of Technology\\ Kanpur\\ Uttar Pradesh - 208016\\ India}
\email{bidyut@iitk.ac.in}
\author{Arya Vadnere}
\address{Chennai Mathematical Institute\\
Siruseri, Tamil Nadu - 603103\\
India}
\email{aryav@cmi.ac.in}

\maketitle

\begin{abstract}
A pair $(\alpha, \beta)$ of simple closed geodesics on a closed and oriented hyperbolic surface $M_g$ of genus $g$ is called a filling pair if the complementary components of $\alpha\cup\beta$ on $M_g$ are simply connected. The length of a filling pair is defined to be the sum of their individual lengths. In~\cite{Aou}, Aougab-Huang conjectured that the length of any filling pair on $M_{g}$ is at least $\frac{m_{g}}{2}$, where $m_{g}$ is the perimeter of the regular right-angled hyperbolic $\left(8g-4\right)$-gon.

In this paper, we prove a generalized isoperimetric inequality for disconnected regions and we prove the Aougab-Huang conjecture as a corollary.
\end{abstract}

\section{Introduction}
Let $M_g$ be a closed and oriented hyperbolic surface of genus $g$. A pair $(\alpha, \beta)$, of simple closed curves on $M_g$ is called a filling pair if the complement of their union $\alpha\cup \beta$ in $M_g$ is a disjoint union of topological discs. It is assumed that the curves $\alpha$ and $\beta$ are in minimal position, i.e., the geometric intersection number $i(\alpha, \beta)$ is equal to $|\alpha\cap\beta|$ (see Section 1.2.3 in~\cite{Far}).

To a filling pair one can associate a natural number $k$, the number of topological discs in the complement $M_g\setminus (\alpha\cup\beta).$ A filling pair $(\alpha, \beta)$ is minimal when $k=1$. For a minimal filling pair $(\alpha, \beta)$ of $M_g$, the geometric intersection number is given by $i(\alpha, \beta)=2g-1$ (see Lemma 2.1 in~\cite{Aou}).

The set of all closed and oriented hyperbolic surfaces of genus $g\geq 2$, up to isometry, is called the moduli space of genus $g$ and is denoted by $\mathcal{M}_g.$ The length of a filling pair $(\alpha, \beta)$ on a hyperbolic surface $M_g\in \mathcal{M}_g$ is defined by the sum of their individual lengths: $$L_{M_g}(\alpha, \beta) = l_{M_g}(\alpha)+ l_{M_g}(\beta),$$ where $l_{M_g}(\alpha)$ denotes the length of the geodesic representative in the free homotopy class $[\alpha]$ of $\alpha$ on $M_g$.

If $(\alpha, \beta)$ is a filling pair of a hyperbolic surface $M_g\in \mathcal{M}_g$, then we assume that $\alpha$ and $\beta$ are simple closed geodesics. When we cut $M_g$ open along a minimal filling pair, we obtain a hyperbolic $(8g-4)$-gon with area $4\pi(g-1)$ which is equal to the area of the surface $M_g$. The length of the filling pair is equal to half of the perimeter of this $(8g-4)$-gon.

It is known that among all hyperbolic $n$-gons with a fixed area, the regular $n$-gon has the least perimeter (see Bezdek~\cite{Bez}). In particular, we see that a regular right-angled $(8g-4)$-gon, denoted by  $\mathcal{P}_g$, has the least perimeter among all $(8g-4)$-gons with fixed area $4\pi(g-1)$. Thus, if $m_g$ is the perimeter of a hyperbolic regular right-angled $(8g-4)$-gon and  $$\mathcal{F}_g(M_g) = \min\{L_{M_g}(\alpha, \beta)\mid (\alpha, \beta) \text{ is a minimal filling of } M_g\},$$ then $$\mathcal{F}_g(M_g)\geq \frac{m_g}{2}, \text{ for all } M_g\in \mathcal{M}_g.$$ This fact is observed in~\cite{Aou} (see Theorem 1.3~\cite{Aou}). It is also shown in~\cite{Aou} that there are finitely many surfaces where the equality holds. Furthermore, Aougab and Huang have defined the filling pair systole function $\mathcal{Y}_g: \mathcal{M}_g\to \mathbb{R}$, by
\begin{equation}\label{systole}
\mathcal{Y}_g(M_g)=\min\{L_{M_g}(\alpha, \beta)\mid (\alpha, \beta) \text{ is a filling pair of } M_g\},
\end{equation}
and conjectured that (see Conjecture 4.6~\cite{Aou}): $$\mathcal{Y}_g(M_g)\geq \frac{m_g}{2}, \text{ for all } M_g\in \mathcal{M}_g.$$ Aougab and Huang have proved their conjecture when $M_g\setminus (\alpha\cup\beta)$ has two components (see Corollary 4.5 in \cite{Aou}).

If $(\alpha, \beta)$ is a filling pair of $M_g\in \mathcal{M}_g$, then the complement $M_g\setminus (\alpha\cup\beta)$ is a collection of even sided polygons with areas that sum to the area of $M_g$ which is equal to $4\pi(g-1)$. The number of sides in each complementary polygon is at least four, which follows from the fact that $\alpha$ and $\beta$ are a pair of simple curves in minimal position. If there are $k$ polygons in $M_g\setminus (\alpha\cup\beta)$, then a calculation with Euler characteristic implies that $$8g-4 = 4(1-k)+2 \sum\limits_{i=1}^k m_i,$$ where $2m_1, \dots, 2m_k,$ are the number of sides of the polygons.

For a polygon $P$, the area and the perimeter of $P$ are denoted by $\mathrm{area}(P)$ and $\mathrm{Perim}(P)$ respectively. 

In this article, we prove the theorem below:
\begin{theorem}[Main Theorem]\label{thm:main}
Suppose $P_i$'s are hyperbolic $2m_i$-gons, with $m_i\geq 2$, for $ i=1, \dots, k$. Let $\mathcal{R}$ be a regular $N$-gon, such that
\begin{enumerate}
\item  $N=4(1-k)+2\sum\limits_{i=1}^km_i$ and 
\item $\mathrm{area}(\mathcal{R}) = \sum\limits_{i=1}^k \mathrm{area}(P_i)$.
\end{enumerate}
If $\mathcal{R}$ is not acute, then $\sum\limits_{i=1}^k \mathrm{Perim}(P_i)\geq \mathrm{Perim}(\mathcal{R}).$
\end{theorem}
As a consequence of Theorem~\ref{thm:main}, we have the corollary below that proves the conjecture of Aougab and Huang.
\begin{cor}\label{cor:conj}
Let $\mathcal{Y}_g$ be the filling pair systole function on $\mathcal{M}_g$, defined in equation~\eqref{systole}. Then $$\mathcal{Y}_g(M_g)\geq \frac{m_g}{2}, \text{ for all } M_g\in \mathcal{M}_g,$$ where $$m_g=(8g-4) \cosh^{-1} \left(2\cos\left( \frac{2\pi}{8g-4}\right) + 1\right)$$ is the perimeter of a regular right-angled hyperbolic $(8g-4)$-gon.
\end{cor}
While it is not hard to see that $\mathcal{Y}_g$ has a global minimum over $\mathcal{M}_g$, this corollary provides an explicit lower bound.

In~\cite{Aou}, Aougab and Huang have proved that $\mathcal{F}_g$ is a topological Morse function. Furthermore, there are finitely many surfaces $M_g$ such that  $\mathcal{F}_g(M_g) = \frac{m_g}{2}$  (for details of the proof, we refer to Theorem 1.3, Section 4~\cite{Aou}). A similar argument proves that $\mathcal{Y}_g$ is a \emph{generalized systole function} (see~\cite{Akr}) and hence a topological Morse function. Furthermore, it follows that there are at most finitely many $M_g\in \mathcal{M}_g$ such that $\mathcal{Y}_g(M_g) = \frac{m_g}{2}$. 

\vspace{0.5 cm}

\noindent \textbf{Acknowledgements:}
The first author would like to thank Siddhartha Gadgil, Mahan Mj and Divakaran D. for all the discussions. The second author would like to thank Satyajit Guin for hosting him at IIT Kanpur, making this work possible. The authors also thank the referee for several helpful comments and suggestions.


\section{Partitions of polygons}
In this section, we develop two lemmas, involving hyperbolic polygons and partitions of their areas, which are essential for the subsequent sections.

Let $\left(\alpha,\beta\right)$ be a filling pair of $M_{g}$. Then the complement of $\alpha\cup\beta$ in $M_{g}$ is a disjoint union of topological discs, and we write 
\begin{equation*}
M_{g}\setminus \left( \alpha \cup \beta \right) = \bigcup\limits_{i=1}^{k} P_{i},
\end{equation*}
where $k\in \mathbb{N}$ and $P_i$, for $i=1,\dots, k$, are topological discs. Note that when $M_{g}$ $\left(g\geq2\right)$ is a hyperbolic surface and $\left(\alpha,\beta\right)$ is a filling pair of geodesics then the polygons $P_{i}$ are hyperbolic polygons. 

 From another point of view, one can regard the union $\Gamma \left( \alpha,\beta\right)=\alpha\cup\beta$ as a decorated fat graph (also known as a ribbon graph) on $M_{g}$, where the intersection points of $\alpha$ and $\beta$ are the vertices, the sub-arcs of $\alpha$ and $\beta$ between the vertices are the edges, and the fat graph structure is determined by the orientation of the surface (we refer to Section 2 in~\cite{Bid} for notations).  Note that $\Gamma(\alpha, \beta)$ is a $4$-regular graph on $M_g$. If the number of vertices and edges in $ \Gamma \left( \alpha, \beta\right)$ are $v$ and $e$ respectively, then we have $e=2v$ and $v=i \left( \alpha, \beta \right)$, where $i \left( \alpha, \beta \right)$ is the geometric intersection number of $\alpha$ and $\beta$. Furthermore, the graph $\Gamma \left( \alpha, \beta \right)$ has $k$ boundary components (or equivalently faces) which is equal to the number of components in $M_{g}\setminus \left( \alpha\cup\beta\right)$. 

It is straightforward to see that $\Gamma \left( \alpha, \beta \right)$ is the $1$-skeleton of a cellular decomposition of $M_{g}$. Therefore by Euler's formula we have  $v-e+k = 2-2g$ which implies
\begin{align*} 
  v &= 2g+k-2 \text{ and}\\
  e &= 4g+2k-4.
\end{align*}
Observe that each edge in $\Gamma \left( \alpha, \beta \right)$ contributes two sides in the set of polygons $P_{i}$, for $i=1,\dots,k$. Among every two consecutive edges of $P_{i}$, one comes from $\alpha$ and the other from $\beta$. Furthermore, the curves $\alpha$ and $\beta$ are in minimal position, i.e., they do not form a bi-gon on $M_{g}$. Therefore, the number of sides of each $P_{i}$ is even and at least four. We assume that the number of sides of $P_{i}$ is $2m_{i}$, for some $m_{i}\geq2$, for $i=1,\dots,k$. Therefore by Euler's formula we have 
\begin{equation}\label{eq2.1}
\sum_{i=1}^{k}m_{i}=4g+2k-4.
\end{equation}
Suppose $\mathcal{P}_g$ is a right-angled regular hyperbolic $\left(8g-4\right)$-gon. Then by the Gauss-Bonnet formula (see Theorem 1.1.7 in~\cite{Bus}), we have $\mathrm{area}\left(\mathcal{P}_g\right)=\mathrm{area}\left(M_{g}\right)=4\pi\left(g-1\right)$.
Thus, 
\begin{equation}\label{eq2.2}
\sum_{i=1}^{k}\mathrm{area}\left(P_{i}\right)=\mathrm{area}\left(\mathcal{P}_g\right).
\end{equation}
In the following lemma we explore implications for the angles of $P_{i}$ in the simplest nontrivial case of equation (\ref{eq2.2}), namely $k=2$.
\begin{lemma}\label{lem:1}
Let $P, P_1, P_2$ be regular hyperbolic $2n$-, $2m_1$-, $2m_2$-gons with interior angles $\theta, \theta_1, \theta_2$ respectively, and suppose that 
\begin{enumerate}
\item $\theta\geq \frac{\pi}{2}$, and $m_1, m_2\geq 2$,
\item $\mathrm{area}(P)=\mathrm{area}(P_1)+\mathrm{area}(P_2)$ and
\item $2m_1+2m_2=2n+4.$
\end{enumerate}
Then we have:
\begin{enumerate}[label=(\alph*)]
\item If $\theta_1\leq \theta_2$, then $\theta_1\leq \theta$ and $\theta_2\geq \frac{\pi}{2}$.
\item If $\theta_1\leq \theta$, then $\theta_2\geq \frac{\pi}{2}$.
\end{enumerate}
\end{lemma}
\begin{proof}
From condition $(3)$, we have $m_{1}+m_{2}=n+2$. Now, using condition $(2)$ and the Gauss-Bonnet theorem, we have 
\begin{align*}
 \sum_{i=1}^{2}[\left(2m_{i}-2\right)\pi-2m_{i}\theta_{i}] &= \left(2n-2\right)\pi-2n\theta\\ \implies  \left(m_{1}+m_{2}\right)\pi-2\pi-\left(m_{1}\theta_{1}+m_{2}\theta_{2}\right) &= \left(n-1\right)\pi-n\theta\\ \implies  m_{1}\theta_{1}+m_{2}\theta_{2} &= n\theta+\pi.
\end{align*}
(a) Consider $\theta_1\leq \theta_2$. Assume that $\theta<\theta_{1}$. This implies $\theta<\theta_{2},$ as $\theta_{1}\leq\theta_{2}$. Now, we have $ \pi+n\theta = m_{1}\theta_{1}+m_{2}\theta_{2} > \left(m_{1}+m_{2}\right)\theta=\left(n+2\right)\theta$, which implies $\theta < \frac{\pi}{2}$. This contradicts condition (1) that $\theta\geq \frac{\pi}{2}$. Thus, we conclude that $\theta_{1}\leq\theta$.

Now, 
\begin{align*}
\pi+n\theta &= m_{1} \theta_{1} + m_{2} \theta_{2} \leq \left( n+2 \right) \theta_{2}\\ \implies  \left(n+2\right)\theta_2 & \geq \pi+\frac{n\pi}{2} = \frac{\pi}{2}\left(n+2\right)\\ \implies\theta_{2} &\geq \frac{\pi}{2}.
\end{align*}
(b) If $\theta_1\leq \theta_2$, then the assertion directly follows from (a).
In the remaining case, assume $\theta_2\leq \theta_1$. By switching the role of $\theta_1$ and $\theta_2$ in (a), we have $\frac{\pi}{2}\leq \theta_1.$ Towards contradiction, if $\theta_2< \frac{\pi}{2}$ and $\theta_1\leq \theta$, then $\theta_2<\frac{\pi}{2}\leq \theta_1 \leq \theta.$ But, we have $\pi+n\theta = m_1\theta_1+m_2\theta_2 < m_1\theta +m_2\frac{\pi}{2}$, which implies $(n-m_1)\theta < (m_2-2)\frac{\pi}{2}$. This implies $\theta < \frac{\pi}{2}$  as $m_2-2=n-m_1$, which contradicts condition (1) that $\theta\geq \frac{\pi}{2}.$ 
\end{proof}
In the next lemma (Lemma~\ref{lem:2}), we generalize Lemma~\ref{lem:1}. Suppose $P_{i}$'s are regular hyperbolic $2m_{i}$-gons, for $m_i\geq 2$, where $i=1,\dots,k$, and $P$ is a regular hyperbolic $2n$-gon with interior angle $\theta\geq \frac{\pi}{2}$, such that 
\begin{equation}\label{2.4}
\sum_{i=1}^{k}\mathrm{area}\left(P_{i}\right)=\mathrm{area}\left(P\right)\text{ and}
\end{equation}
\begin{equation}\label{2.5}
2n=4\left(1-k\right) +2\sum_{i=1}^{k}m_{i}.
\end{equation}
Suppose the interior angles of $P_{i}$'s are $\theta_{i}$, for $i=1,\dots,k$.
We define 
\begin{align*}
\theta_{\text{min}}  &=\min\left\{ \theta_{i}\mid i=1,\dots,k\right\}\text{ and} \\
\theta_{\max}  &=\max\left\{ \theta_{i}\mid i=1,\dots,k\right\}. 
\end{align*}

\begin{lemma}\label{lem:2} In the setting above, we have 
\begin{enumerate}
\item $\theta_{\min}\leq\theta$ and \item $\theta_{\max}\geq\frac{\pi}{2}$.
\end{enumerate}
\end{lemma}
\begin{proof}
The proof of Lemma~\ref{lem:2} is similar to the proof of Lemma~\ref{lem:1}. By the Gauss-Bonnet formula, equations~\eqref{2.4} and~\eqref{2.5}, we have 
\begin{align*}
  \sum_{i=1}^{k}\left[ \left( 2m_{i} - 2 \right) \pi - 2m_{i} \theta_{i} \right] &= \left(2n-2\right)\pi-2n\theta\\
\implies  \pi\sum_{i=1}^{k}m_{i}-k\pi-\sum_{i=1}^{k}m_{i}\theta_{i} &= \left(n-1\right)\pi-n\theta\\
\implies  n\pi+\left(2k-2\right)\pi-k\pi-\sum_{i=1}^{k}m_{i}\theta_{i} &= n\pi-\pi-n\theta\\
\implies  \sum_{i=1}^{k}m_{i}\theta_{i} &= n\theta+k\pi-\pi.
\end{align*}

\begin{enumerate}
\item Using the inequality  $\theta_{\min} \left(\sum\limits_{i=1}^{k}m_{i}\right)\leq\sum\limits_{i=1}^{k}m_{i}\theta_{i}$ and equation~\eqref{2.5}, we have 
\begin{align*}
\theta_{\min}\left(n+2k-2\right) &\leq n\theta+\left(k-1\right)\pi\\ &\leq
n\theta+\left(2k-2\right)\theta\\ \implies\theta_{\min} &\leq \theta.
\end{align*}

\item Similarly, using the inequality $\theta_{\max}\left(\sum\limits_{i=1}^{k}m_{i}\right)\geq\sum\limits_{i=1}^{k}m_{i}\theta_{i}$ and equation~\eqref{2.5}, we have 
\begin{align*}
\theta_{\max}\left(n+2k-2\right) &\geq n\theta+\left(2k-2\right)\frac{\pi}{2}\\ &\geq \left(n+2k-2\right)\frac{\pi}{2} \\ \implies\theta_{\max} &\geq \frac{\pi}{2}.
\end{align*}
\end{enumerate}
This completes the proof.
\end{proof}
Now, we note that the proposition (Proposition~\ref{Thm1}) below is the key step in proof of the main theorem (Theorem~\ref{thm:main}).
\begin{prop}\label{Thm1}
Let $P$ be a regular hyperbolic $2n$-gon with interior angle $\theta\geq\frac{\pi}{2}$. Suppose $P_{i}$'s are regular hyperbolic $2m_{i}$-gons, for $m_i\geq 2$ and $i=1,2$, such that
\begin{enumerate}
\item $m_{1}+m_{2}=n+2$ and 
\item $\mathrm{area}\left(P_{1}\right)+\mathrm{area}\left(P_{2}\right)=\mathrm{area}(P)$.
\end{enumerate}
Then $$\mathrm{Perim}(P)\leq\mathrm{Perim}\left(P_{1}\right)+\mathrm{Perim}\left(P_{2}\right).$$ 
\end{prop}
We will prove the main theorem in  \cref{sec:3}, assuming Proposition \ref{Thm1}. The proof of Proposition \ref{Thm1} can be found in \cref{sec:7}, after building up the requisite analysis in  \cref{sec:4}-\cref{sec:6}.

\section{Proof of Main Theorem}\label{sec:3}
In this section, we show that Proposition~\ref{Thm1} implies the main theorem (Theorem~\ref{thm:main}). 
Suppose $\mathcal{P}_g$ is a regular right-angled hyperbolic $(8g-4)$-gon and $P_{i}$'s are regular $2m_{i}$-gons, for $m_i\geq 2$, where $i=1, \dots, k$, satisfying equations~\eqref{2.4} and ~\eqref{2.5}. We prove the theorem stated below:
\begin{theorem}\label{prop:3.1}
$\mathrm{Perim}\left(\mathcal{P}_g\right)\leq\sum\limits_{i=1}^{k}\mathrm{Perim}\left(P_{i}\right).$
\end{theorem}
In light of the fact that the regular $n$-gon has the least perimeter among all hyperbolic $n$-gons with a fixed area (Bezdek~\cite{Bez}), Theorem~\ref{prop:3.1} implies Theorem~\ref{thm:main}.
\begin{proof}[Proof of Theorem~\ref{prop:3.1}]
Suppose $\theta_{i}$'s are the interior angles of $P_{i}$, for $i=1,\dots,k$. After re-indexing, if needed, we assume that  $$\theta_{1}\geq\theta_{2}\geq\dots\geq\theta_{k}.$$
We define regular hyperbolic $2\widetilde{m}_{j}$-gons $\widetilde{P}_{j}$, for $j=1,\dots,k,$ inductively as described below: 
\begin{enumerate}
\item For $j=1$, we have that $\widetilde{P}_{1}=P_{1}$. Here, $\widetilde{m}_{1}=m_{1}$ and $\mathrm{area}\left(\widetilde{P}_{1}\right)=\mathrm{area}\left(P_{1}\right)$. 
\item For $j\geq2$, the polygon $\widetilde{P}_j$ is defined by the requirements that $2\widetilde{m}_{j}=2\widetilde{m}_{j-1}+2m_{j}-4$ and $\mathrm{area}\left(\widetilde{P}_{j}\right)=\mathrm{area}\left(\widetilde{P}_{j-1}\right)+\mathrm{area}\left(P_{j}\right)$.
\end{enumerate}
Now, we prove Lemma~\ref{lem3.2} below which is used to complete the proof of Theorem~\ref{prop:3.1}.
\begin{lemma}\label{lem3.2}
The interior angle $\widetilde{\theta}_{j}$ of $\widetilde{P}_{j}$ satisfies $\widetilde{\theta}_{j}\geq\frac{\pi}{2}$, for each $1\leq j\leq k$. 
\end{lemma}
\begin{proof}[Proof of Lemma~\ref{lem3.2}]
The proof is by induction on $j$. \\
For the base case $j=k$, it is straightforward to see that $\widetilde{P}_{k}=\mathcal{P}_g$, as $\mathrm{area}\left(\widetilde{P}_{k}\right)=\mathrm{area}\left(\mathcal{P}_g\right)$ and $2\widetilde{m}_{k}=\left(\sum\limits_{i=1}^{k}2m_{i}\right)-4\left(k-1\right)$ which is equal to $(8g-4)$. So, the polygon $\widetilde{P}_{k}$ is isometric to $\mathcal{P}_g$ and by the hypothesis, we have $\widetilde{\theta}_{k}=\theta\geq\frac{\pi}{2}.$

Now, assume that the lemma is true for $k_0$, i.e. $\widetilde{\theta}_{k_{0}} \geq \frac{\pi}{2}$, for some $k_{0}\leq k$. 

To complete the induction, we show that $\widetilde{\theta}_{k_{0}-1} \geq \frac{\pi}{2}$. First, note that the polygons $P_{1},\dots,P_{k_{0}}$ and $P:=\widetilde{P}_{k_{0}}$ satisfy the conditions of Lemma~\ref{lem:2}:
\begin{enumerate}
\item The interior angle $\widetilde{\theta}_{k_{0}}$ of $P=\widetilde{P}_{k_0}$ satisfies $\widetilde{\theta}_{k_{0}}\geq\frac{\pi}{2}$. 
\item By definition of $\widetilde{P}_j$'s, we have   $\sum\limits_{i=1}^{k_{0}}\mathrm{area}\left(P_{i}\right)=\mathrm{area}\left(\widetilde{P}_{k_{0}}\right)$.
\item As $2\widetilde{m}_{j}=2\widetilde{m}_{j-1}+2m_{j}-4$, for $j=2,\dots,k_{0}$, we have $2\widetilde{m}_{k_{0}}=4\left(1-k_{0}\right)+2\sum\limits_{i=1}^{k_{0}}m_{i}.$
\end{enumerate}
Now, the definition $\theta_{k_{0}}=\min\left\{ \theta_{i}\mid i=1,\dots,k_{0}\right\}$ and Lemma~\ref{lem:2} imply $\theta_{k_{0}}\leq\widetilde{\theta}_{k_{0}}$.  \\ Finally, the polygons $\widetilde{P}_{k_{0}}, P_{k_{0}}$ and $\widetilde{P}_{k_{0}-1}$ satisfy the following:
\begin{enumerate}
\item The interior angle $\widetilde{\theta}_{k_{0}}$ of $\widetilde{P}_{k_{0}}$ satisfies $\widetilde{\theta}_{k_{0}}\geq\frac{\pi}{2}$,

\item $\mathrm{area}\left(\widetilde{P}_{k_{0}-1}\right)+\mathrm{area}\left(P_{k_{0}}\right)=\mathrm{area}\left(\widetilde{P}_{k_{0}}\right)$,

\item $2\widetilde{m}_{k_{0}-1}+2m_{k_{0}}=2\widetilde{m}_{k_{0}}+4$ and

\item $\theta_{k_{0}}\leq\widetilde{\theta}_{k_{0}}$. 
\end{enumerate}
Thus, by Lemma~\ref{lem:1}, we conclude that $\widetilde{\theta}_{k_{0}-1}\geq\frac{\pi}{2}$.
\end{proof}
Now, we complete the proof of Theorem~\ref{prop:3.1}. By Lemma~\ref{lem3.2}, the polygons $\widetilde{P}_{j},P_{j}$ and $\widetilde{P}_{j-1}$, for $2\leq j\leq k$, 
satisfy following:
\begin{enumerate}
\item The interior angle of $\widetilde{P}_{j}$ is $\widetilde{\theta}_{j}\geq\frac{\pi}{2}$,
\item $\mathrm{area}\left(\widetilde{P}_{j-1}\right)+\mathrm{area}\left(P_{j}\right)=\mathrm{area}\left(\widetilde{P}_{j}\right)$ and
\item $2\widetilde{m}_{j-1}+2m_{j}=2\widetilde{m}_{j}+4$.
\end{enumerate}
Thus, by Proposition~\ref{Thm1}, we conclude that 
\begin{equation*}
\mathrm{Perim}\left(\widetilde{P}_{j}\right)\leq\mathrm{Perim}\left(P_{j}\right)+\mathrm{Perim}\left(\widetilde{P}_{j-1}\right),
\end{equation*}
 for $2\leq j\leq k$, which implies
\begin{eqnarray*}
\mathrm{Perim}\left(\mathcal{P}_g\right)=\mathrm{Perim}\left(\widetilde{P}_{k}\right)\leq\sum_{i=1}^{k}\mathrm{Perim}\left(P_{i}\right).
\end{eqnarray*}
\end{proof}

\begin{cor}\label{cor:conj}
Let $M=M_{g}$ be a closed hyperbolic surface of genus $g$ and $\left(\alpha,\beta\right)$
be a filling pair of simple closed geodesics. Then 
\begin{equation*}
L_M\left( \alpha, \beta \right)= l_M\left(\alpha\right)+l_M\left(\beta\right)\geq\frac{m_{g}}{2},
\end{equation*}
 where $m_{g}$ is the perimeter of a regular right-angled hyperbolic $\left(8g-4\right)$-gon.
\end{cor}
\begin{proof}
Let $M\setminus \left( \alpha \cup \beta \right) =\bigcup\limits_{i=1}^{k}\widehat{P}_{i},$ where $\widehat{P}_{i}$'s are hyperbolic $2m_{i}$-gons, for $m_{i}\geq2$, where $i=1, \dots, k$. We denote $P_{i}$ to be a regular hyperbolic $2m_{i}$-gon whose area is equal to $\text{area}\left(\widehat{P}_{i}\right)$. 
Then we have 
\begin{align*}
\mathrm{Perim}\left(P_{i}\right)&\leq\mathrm{Perim}\left(\widehat{P}_{i}\right)\\ \implies \sum_{i=1}^{k}\mathrm{Perim}\left(P_{i}\right)&\leq\sum_{i=1}^{k}\mathrm{Perim}\left(\widehat{P}_{i}\right)=2 L_M\left(\alpha,\beta\right).
\end{align*}
Now, Theorem~\ref{prop:3.1} implies that 
\begin{align*}
m_{g}&\leq\sum_{i=1}^{k}\mathrm{Perim}\left(P_{i}\right)\leq 2 L_M\left(\alpha,\beta\right)\\ \implies L_M\left( \alpha, \beta\right) &= l_M\left(\alpha\right)+l_M\left(\beta\right)\geq\frac{m_{g}}{2}.
\end{align*} 
\end{proof}

Now, we aim at proving Proposition~\ref{Thm1}.


\section{Generalization of the isoperimetric inequality} \label{sec:4}
The purpose of this section is to prove Proposition~\ref{prop:gen_ineq} which is essential in the subsequent sections to prove Proposition~\ref{Thm1}. For $n\geq 3$ and $0<a<(n-2)\pi$, by $P_n(a)$ we denote a regular hyperbolic $n$-gon with area $a$. The perimeter of $P_n(a)$ is given by (for a proof, see \cite{San}): 
\begin{equation}\label{eq:perim}
\mathrm{Perim}\left(P_{n}(a)\right)=2n\cosh^{-1}\left(\frac{\cos\left(\frac{\pi}{n}\right)}{\sin\left(\frac{\left(n-2\right)\pi-a}{2n}\right)}\right).
\end{equation}
For $a=0$, the polygon $P_n(a)$ is degenerate. In this case, the quantity $\mathrm{Perim}\left(P_{n}(a)\right)=0$. Note that for a fixed area $a>0$, the function $\mathrm{Perim}\left(P_{n}(a)\right)$ is strictly decreasing in $n$. Furthermore, for a fixed $n$ the function $\mathrm{Perim}\left( P_{n}(a) \right)$ is strictly increasing in $a$. We prove:
\begin{prop}\label{prop:gen_ineq}
For $n\geq4$, the function $g_n:[0, (n-2)\pi) \to \mathbb{R}$, defined by $$g_{n}(x)=\mathrm{Perim}\left(P_{n}(x)\right)-\mathrm{Perim}\left(P_{n+1}(x)\right),$$
is monotonically increasing in $x$.
\end{prop}
Now, we develop two technical lemmas which will be used in the proof of Proposition~\ref{prop:gen_ineq}.
\begin{lemma}\label{lem:equation_5}
Let $x\in\left(0,\frac{\pi}{4}\right]$ and  $y\in\left(x,\frac{\pi}{2}\right)$. Then, we have 
\begin{equation}\label{eq:x=000026y}
1+ \sin^{2}y<\frac{\cos^{2} x }{\cos^{2} y }+ \frac{x \tan x}{y \tan y } \sin^{2} y.
\end{equation}
\end{lemma}

\begin{proof}
Note that $\lim\limits_{x\to y} \left(\frac{\cos^{2} x }{\cos^{2} y }+ \frac{x \tan x}{y \tan y } \sin^{2} y\right) = 1+\sin^2 y$. Therefore, to prove inequality~\eqref{eq:x=000026y}, it suffices to show $\frac{\cos^{2} x }{\cos^{2} y }+ \frac{x \tan x}{y \tan y } \sin^{2} y$ is monotonically decreasing in $x$ on $(0, \min\{\pi/4, y\})$, when $y$ is fixed. Equivalently, we show
\begin{align}
\frac{\partial}{\partial x}\left( \frac{\cos^{2}x}{\cos^{2}y}+ \frac{x\tan x}{y \tan y}\sin^{2} y \right) & < 0 \nonumber\\
\iff -\frac{\sin2x}{\cos^{2}y} + \frac{\sin y \cos y}{y} (\tan x+ x\sec^2 x) & < 0\nonumber \\
\iff \frac{\tan x+x\sec^{2}x}{\sin 2x} & <\frac{y}{\sin y\cos^{3} y }.\label{eq:der_x}
\end{align}
Now, note that $\frac{\tan x + x \sec^{2} x }{\sin 2x}$ is monotonically increasing on $\left(0,\frac{\pi}{2}\right)$ and $x<y$.
Therefore, it suffices to show the inequality below holds true:
\begin{align*}
\frac{\tan y + y\sec^{2}y}{2\sin y \cos y } & \leq \frac{y}{\sin y \cos^{3} y }\\
\iff \frac{\tan y }{2} + \frac{y}{2\cos^{2} y } & \leq \frac{y}{\cos^{2} y }\\
\iff \sin y \cos y  & \leq y,
\end{align*}
which is true for every $y\in\left(0,\frac{\pi}{2}\right)$. Hence, the proof is complete.
\end{proof}
\begin{lemma}\label{lem:der_inc}
For $x>0$, the function $H_x : \left(c_{x},\infty\right) \to \mathbb{R}$, defined by $$H_x\left(t\right)=\frac{\cos^2\left(\frac{\pi}{t}\right) \tan^2\left(\frac{c_{x}\pi}{2t}\right)}{\cos^{2}\left(\frac{\pi}{t}\right)-\cos^{2}\left(\frac{c_{x}\pi}{2t}\right)},$$ is monotonically decreasing in $t$, where $c_{x}=2+\frac{x}{\pi}$.
\end{lemma}

\begin{proof}
We show that $\frac{d}{dt}H_x(t) \leq0$. We have
\begin{align*}
H_x\left(t\right) &= \frac{\cos^2\left(\frac{\pi}{t}\right)\tan^2\left(\frac{c_x \pi}{2t}\right)}{\cos^{2}\left(\frac{\pi}{t}\right)-\cos^{2}\left(\frac{c_x\pi}{2t}\right)}= \frac{\tan^2\left(\frac{c_x\pi}{2t}\right)}{1-\frac{\cos^{2}\left(c_x\pi/2t\right)}{\cos^2\left(\pi/t \right)}}\\ \implies \frac{d}{dt}H_x(t) &= -\frac{2\pi c_x \tan\left(\frac{c_x \pi }{t} \right) \sec^{2}\left(\frac{c_x \pi }{t}\right)}{t^{2} \left( 1 - \frac{\cos^{2}\left(c_x \pi / t\right)}{\cos^{2}\left(\pi/t\right)}\right)} - \frac{2 \pi \tan^{2}\left(\frac{c_x \pi}{t}\right) \tan\left(\frac{\pi}{t}\right)\sec^{2}\left(\frac{\pi}{t}\right)\cos^{2}\left(\frac{c_x \pi }{t}\right)}{t^{2}\left(1-\frac{\cos^{2}\left(c_x\pi/t\right)}{\cos^{2}\left(\pi/t\right)}\right)^{2}}\\ & \ \ \ + \frac{2\pi c_x\tan^{2}\left(\frac{c_x\pi}{t}\right)\sec^{2}\left(\frac{\pi}{t}\right)\sin\left(\frac{c_x\pi}{t}\right)\cos\left(\frac{c_x\pi}{t}\right)}{t^{2}\left(1-\frac{\cos^{2}\left(c_x\pi/t\right)}{\cos^{2}\left(\pi/t\right)}\right)^{2}}.
\end{align*}
Therefore, we have that $ \frac{d}{dt}H_x(t) \leq 0$ if and only if $$c_x \sec^{2}\left(\frac{c_x \pi }{t}\right) + \sin\left(\frac{c_x\pi}{t}\right) \tan\left(\frac{\pi}{t}\right)\sec^{2}\left(\frac{\pi}{t}\right)\cos\left(\frac{c_x\pi}{t}\right) \geq c_x\sec^{2}\left(\frac{\pi}{t}\right)+c_x\sin^{2}\left(\frac{c_x\pi}{t}\right)\sec^{2}\left(\frac{\pi}{t}\right).$$ We define $\alpha=\frac{c_x \pi}{t}$ and $\beta=\frac{\pi}{t}$.  As $t\geq4$, we get $\beta\in\left( 0,\frac{\pi}{4}\right] $ and $\alpha\in\left(\beta,\frac{\pi}{2}\right)$. In this notation, it suffices to show:
\begin{align}
\alpha \tan \alpha \frac{\cos^{2} \beta}{\cos^{2} \alpha}+\beta \sin^{2}\alpha \tan\beta & \geq \alpha \tan\alpha + \alpha \sin^{2}\alpha\tan\alpha\nonumber \\
\iff \alpha\tan\alpha\left(\frac{\cos^{2}\beta}{\cos^{2}\alpha}+ \frac{\beta\tan(\beta)}{\alpha\tan \alpha }\sin^{2} \alpha \right) & \geq\alpha \tan\alpha\left(1+\sin^{2}\alpha\right)\nonumber \\
\iff \frac{\cos^{2}\beta}{\cos^{2}\alpha}+\frac{\beta\tan\beta}{\alpha\tan\alpha}\sin^{2}\alpha &\geq 1+\sin^{2}\alpha.\label{eq:beta}
\end{align}
By Lemma \ref{lem:equation_5}, we conclude that inequality~\eqref{eq:beta} is true.
\end{proof}
\begin{proof}[Proof of Proposition~\ref{prop:gen_ineq}]
For $x=0$ and $n\geq 4$, we have that $\mathrm{Perim}\left(P_n(x)\right)=0$. Therefore, we see that $g_{n}\left(0\right)=0$. 
Now, it suffices to show that $g_{n}'(x) \geq0$ for $x \in\left(0,\left(n-2\right)\pi\right)$. 

We define $h_n(x) = \frac{d}{dx}\mathrm{Perim}\left(P_n(x)\right)$. Then $h_n(x)=\frac{\cos\left(\frac{\pi}{n}\right)\tan\left(\frac{2\pi+x}{2n}\right)}{\sqrt{\cos^{2}\left(\frac{\pi}{n}\right)-\cos^{2}\left(\frac{2\pi+x}{2n}\right)}}.$ Now, we show that $g_{n}'(x)=h_n(x)-h_{n+1}(x) \geq0$, where $x \in\left(0,\left(n-2\right)\pi\right)$. Therefore, it suffices to show that for an arbitrary but fixed $x$, the function $h_n\left(x\right)$ decreases with $n$. Now, the function $h_n(x)$ decreases with $n$ if and only if $\left(h_n(x)\right)^2$ decreases with $n$, as $h_n(x)\geq0$ by the isoperimetric inequality. 

For a fixed $x$, the function $H_x: \left(2+\frac{x}{\pi},\infty\right) \to \mathbb{R}$ by $$H_x\left(t\right)=\left(\frac{\cos\left(\frac{\pi}{t}\right)\tan\left(\frac{2\pi+x}{2t}\right)}{\sqrt{\cos^{2}\left(\frac{\pi}{t}\right)-\cos^{2}\left(\frac{2\pi+x}{2t}\right)}}\right)^{2}.$$ 
is monotonically decreasing by Lemma~\ref{lem:der_inc}. Therefore $\left(h_n(x)\right)^2$,  which is the restriction of $H_x$ to $\mathbb{N}\cap\left(2+\frac{x}{\pi},\infty\right)$, is decreasing in $n$.
\end{proof}
We conclude this section by the corollary below:
\begin{cor}\label{cor:min_at_0}
Let $m\geq2$ and $a\in \left(0,(2m-2)\pi\right)$ be fixed. Consider the family of functions $f_{m,n}:[0,a)\to\mathbb{R}$, defined by  $$ f_{m,n}(x)=\mathrm{Perim}\left(P_{2m}(x)\right)+\mathrm{Perim}\left(P_{2\left(n+2-m\right)}\left(a-x\right)\right),$$ for $n\geq 2m$. If $f_{m,n_{0}}$ admits its minimum at $x=0$, for some $n_{0}$, then $f_{m,k}$ admits its minimum at $x=0$, for every $k\geq n_{0}$.
\end{cor}

\begin{proof}
The proof is by induction. The base case $k=n_{0}$ is the hypothesis of the corollary. Assume for some $k\geq n_{0}$, we have that $f_{m,k}(0)=\min\left\{ f_{m,k}(x)\mid x\in[0,a)\right\} $. Now, Proposition \ref{prop:gen_ineq} implies that $\mathrm{Perim}\left(P_{2\left(k+2-m\right)}\left(a-x\right)\right)-\mathrm{Perim}\left(P_{2\left(k+3-m\right)}(a-x)\right)$ admits maximum at $x=0$. Therefore, the function $f_{m,k}(x)-f_{m,k+1}(x)$ admits maximum at $x=0$. Thus, we have
\begin{align*}
f_{m,k+1}\left(x\right) & =f_{m,k}(x)-\left(f_{m,k}(x)-f_{m,k+1}(x)\right)\\
 & \geq f_{m,k}(0)-\left(f_{m,k}(x)-f_{m,k+1}(x)\right)\\
 & \geq f_{m,k}(0)-\left(f_{m,k}(0)-f_{m,k+1}(0)\right)=f_{m,k+1}(0).
\end{align*}
This completes the proof.
\end{proof}


\section{Base Cases $n=2,3$}
In this section, we prove Proposition~\ref{Thm1} for the cases: $n=2$ and $3$ in Lemma~\ref{n=2,3}. First, we develop Lemma~\ref{lem:iso_ineq_cont}, next recall Theorem~\ref{thm:iso_ineq_discon}  and then finally we prove Lemma~\ref{n=2,3}. 

\begin{lemma} \label{lem:iso_ineq_cont}
For $t>0$, the function $p_{t}(x):\left(\frac{t}{\pi}+2,\infty\right)\to\mathbb{R}$, defined by $$p_{t}(x)=2x\cosh^{-1}\left(\frac{\cos\left(\frac{\pi}{x}\right)}{\sin\left(\frac{\left(x-2\right)\pi-t}{2x}\right)}\right),$$
 is decreasing in $x$. In particular, the function $\mathrm{Perim}\left(P_n\left(t\right)\right)$ is decreasing in $n$, for $n\geq3$.
\end{lemma}
\begin{proof}
Let $\Omega=\left\{ \left(x,y\right)\in\mathbb{R}^{2}\mid x>2,0\leq y<\left(x-2\right)\pi\right\} $. Consider the function $F:\Omega \to \mathbb{R}$, defined by $$F\left(x,y\right)=2x\cosh^{-1}\left(\frac{\cos\left(\frac{\pi}{x}\right)}{\sin\left(\frac{\left(x-2\right)\pi-y}{2x}\right)}\right),$$ so that $p_{t}(x)=F(x,t)$.
The function $F$ is smooth in the interior of $\Omega$. Hence, by Lemma~\ref{lem:der_inc}, we have
\begin{align*}
\frac{\partial}{\partial y}\frac{\partial}{\partial x}F(x_0,y_0) =\frac{\partial}{\partial x}\frac{\partial}{\partial y}F(x_0,y_0) <0,
\end{align*}
where $(x_0, y_0)$ is an interior point of $\Omega$. Thus, for a fixed $x_{0}> 2$, the function $\frac{\partial}{\partial x}F\left(x_0,y\right)$ is decreasing in $y\in \left(0,\left(x_{0}-2\right)\pi\right)$. Now, $F\left(x,0\right)=0$ implies $\lim\limits_{y\to0^{+}}\frac{\partial}{\partial x}F\left(x,y\right)=0$. \\ Hence, for every $t>0$ and $x_0\in \left(\frac{t}{\pi}+2,\infty\right)$, we have $$p'_{t}\left(x_{0}\right)=\frac{\partial}{\partial x}F(x_0,t )<0,$$ proving $p_{t}(x)$ is decreasing in $x$.

Finally, we have that $\mathrm{Perim}\left(P_n\left(t\right)\right) = p_{t}(n)$ implies  $\mathrm{Perim}\left(P_n\left(t\right)\right)$ is decreasing in $n\geq 3$.
\end{proof}
\begin{theorem}~\cite{San}\label{thm:iso_ineq_discon} 
Let $P,P_{1}$ and $P_{2}$ be regular hyperbolic $k$-gons, for $k\geq3$, with $\mathrm{area}(P)=\mathrm{area}(P_1)+\mathrm{area}(P_2)$. If the interior angle $\theta$ of $P$ satisfies $\theta\geq\cos^{-1}\left(-1+2\sin\left(\pi/k\right)\right)$, then $$ \mathrm{Perim}\left(P_{1}\right)+\mathrm{Perim}\left(P_{2}\right)\geq\mathrm{Perim}\left(P\right).$$
\end{theorem}

\begin{lemma}\label{n=2,3}
Proposition~\ref{Thm1} is true for the cases: $n=2$ and $3$.
\end{lemma}

\begin{proof} 
Proposition \ref{Thm1} is vacuously true for $n=2$, as a regular hyperbolic quadrilateral with interior angle at least $\frac{\pi}{2}$ must be degenerate (with area $0$). Consider the case $n=3$. Assume that $m_{1}\leq m_{2}$ which implies $m_{1}=2$ and $m_{2}=3$. If $Q$ is the regular hyperbolic $6$-gon with area equal to $\mathrm{area}\left(P_{1}\right)$, then $\mathrm{Perim}\left(P_{1}\right)\geq\mathrm{Perim}\left(Q\right)$ (by Lemma~\ref{lem:iso_ineq_cont}). Furthermore, note that $\frac{\pi}{2}=\cos^{-1}\left(-1+2\sin\left(\frac{\pi}{6}\right)\right)$. So, Theorem~\ref{thm:iso_ineq_discon} implies $\mathrm{Perim}\left(P\right)\leq\mathrm{Perim}\left(Q\right)+\mathrm{Perim}\left(P_{2}\right).$ Hence, $$\mathrm{Perim}\left(P_{1}\right)+\mathrm{Perim}\left(P_{2}\right)\geq\mathrm{Perim}\left(Q\right)+\mathrm{Perim}\left(P_{2}\right)\geq\mathrm{Perim}\left(P\right).$$
\end{proof}
%


\section{Case $m_{1}=2$} \label{sec:6}

In this section, we prove Proposition~\ref{prop:m1=00003D2} and then as a corollary, we prove Proposition \ref{Thm1} for the case $m_{1}=2$. The case $m_2=2$ follows similarly by interchanging the role of $m_1$ and $m_2$. Before we proceed, recall that in Proposition \ref{Thm1}, it is assumed the interior angle $\theta$ of the $2n$-gon $P$ satisfies $\theta\geq\frac{\pi}{2}$. By the Gauss-Bonnet formula, this translates to $\text{area}\left(P\right)=a\leq\left(n-2\right)\pi$.

\begin{prop}\label{prop:m1=00003D2}
Let $m\geq4$ and $a\in (0,\left(m-2\right)\pi)$ be fixed. Consider $b=\min\{2\pi, a\}$. Then the function $f_{2,m}:\left[0, b\right)\to\mathbb{R}$,  defined (as in Corollary \ref{cor:min_at_0}) by $$f_{2,m}\left(x\right)=\mathrm{Perim}\left(P_{4}(x)\right)+\mathrm{Perim}\left(P_{2m}\left(a-x\right)\right),$$ admits the minimum at $x=0$.
\end{prop}

Now, we prove Lemma~\ref{lem:curve_above_chord} and  Lemma~\ref{lem:unique_sol} which are used in the proof of Proposition~\ref{prop:m1=00003D2}.
\begin{lemma}\label{lem:curve_above_chord}
Suppose $u:\left[a,b\right]\to\mathbb{R}$ is a continuous function such that the graph of $u$ does not intersect the chord $\overline{AB}$, joining $\left(a,u(a)\right)$ and $\left(b,u(b)\right)$, in the interior. If $u$ is differentiable on $(a, b)$ and $\lim\limits_{x\to a}u'(x)=+\infty$, then the graph of $u$ lies above the chord $\overline{AB}$. Furthermore, for any $x\in\left(0,b-a\right)$, we have $$u(a+x)+u\left(b-x\right)>u(a)+u(b).$$
\end{lemma}
\begin{proof}
We define a new function $v:[a,b]\to\mathbb{R}$ by $v(x)=u(a)+\frac{u(b)-u(a)}{b-a} (x-a)$. Then the graph of $v$ is $\overline{AB}$. By the condition $\lim\limits_{x\to a}u'(x)=+\infty$, we have  $u(x)>g (x)$, where $ 0< x-a< \epsilon$, for some $\epsilon> 0$.
Now, if $u(x')\leq v(x')$ for some $x' \in (a, b)$,  then Intermediate Value Theorem implies $u(x_0)=v(x_0)$, for some $x_0\in\left(a,b\right)$. This contradicts that the graph of $u$ does not intersect the chord $\overline{AB}$ in the interior.  Thus, the graph of $u$ lies above $\overline{AB}$, and the inequality follow. 
\end{proof}

\begin{lemma} \label{lem:unique_sol}
Let $\phi:\left(0,6\pi\right)\to\mathbb{R}$ be defined by 
\[
\phi(x)=\mathrm{Perim}\left(P_{8}(x)\right)=16\cosh^{-1}\left(\frac{\cos\left(\frac{\pi}{8}\right)}{\sin\left(\frac{6\pi-x}{16}\right)}\right).
\]
Then the function $\Tilde{\phi}(x)=x\phi'(x)-\phi(x)$ has a unique root in $\left(0,6\pi\right)$.
\end{lemma}

\begin{proof}
We first show that $\Tilde{\phi}$ has a unique extremum in $\left(0,6\pi\right)$ (which would imply that $\Tilde{\phi}$ has at most one root) and then show that $\Tilde{\phi}$ has exactly one root. Now, a local extremum for $\Tilde{\phi}$ only occurs at roots of $\phi''$, since $\Tilde{\phi}'(x)=x\phi''(x)$. Now,
\begin{align*}
\phi''(x) & =\frac{\cos\left(\frac{\pi}{8}\right)\csc^{2}\left(\frac{6\pi-x}{16}\right)}{16\sqrt{\cos^{2}\left(\frac{\pi}{8}\right)-\sin^{2}\left(\frac{6\pi-x}{16}\right)}}-\frac{\cos\left(\frac{\pi}{8}\right)\cos^{2}\left(\frac{6\pi-x}{16}\right)}{16\left(\cos^{2}\left(\frac{\pi}{8}\right)-\sin^{2}\left(\frac{6\pi-x}{16}\right)\right)^{3/2}}.
\end{align*}
So $\phi''(x)=0$ if and only if
\begin{align*}
\frac{\cos\left(\frac{\pi}{8}\right)\csc^{2}\left(\frac{6\pi-x}{16}\right)}{16\sqrt{\cos^{2}\left(\frac{\pi}{8}\right)-\sin^{2}\left(\frac{6\pi-x}{16}\right)}} & =\frac{\cos\left(\frac{\pi}{8}\right)\cos^{2}\left(\frac{6\pi-x}{16}\right)}{16\left(\cos^{2}\left(\frac{\pi}{8}\right)-\sin^{2}\left(\frac{6\pi-x}{16}\right)\right)^{3/2}}\\
\iff\cos^{2}\left(\frac{\pi}{8}\right)-\sin^{2}\left(\frac{6\pi-x}{16}\right) & =\sin^{2}\left(\frac{6\pi-x}{16}\right)\cos^{2}\left(\frac{6\pi-x}{16}\right)\\
\iff\cos^{2}\left(\frac{\pi}{8}\right) & =\sin^{2}\left(\frac{6\pi-x}{16}\right)\left(1+\cos^{2}\left(\frac{6\pi-x}{16}\right)\right).
\end{align*}

Now, the function $x\mapsto\sin^{2}\left(\frac{6\pi-x}{16}\right)\left(1+\cos^{2}\left(\frac{6\pi-x}{16}\right)\right)$ is a bijection from $\left(0,6\pi\right)$ to its codomain
$\left(0,\sin^{2}\left(\frac{3\pi}{8}\right)\left(1+\cos^{2}\left(\frac{3\pi}{8}\right)\right)\right)$.
We can check that $\cos^{2}\left(\frac{\pi}{8}\right)$ lies in the codomain, so has a unique pre-image. In particular, there is a unique $x\in (0,6\pi)$ such that $\Tilde{\phi}'\left(x\right)=0$. 

Thus, $\Tilde{\phi}$ has a unique local extremum in $\left(0,6\pi\right)$ (one can check that this is a local minimum).
Since $\lim_{x\to0^{+}}\Tilde{\phi}(x)=0$, it follows that $\Tilde{\phi}$
can have at most one root in $(0,6\pi)$. We can check that $9\cdot\phi'(9)-\phi(9)\approx-0.45<0$
and $10\cdot\phi'(10)-\phi(10)\approx1.05>0$ to see that $\Tilde{\phi}$
has at least one root in the interval $(9,10)$, and hence exactly one root in $\left(0,6\pi\right)$.
One can compute the value of this root to be $x\approx9.34$.
\end{proof}

\begin{proof}[Proof of Proposition~\ref{prop:m1=00003D2}]
In light of Corollary~\ref{cor:min_at_0}, it suffices to prove the proposition for the case $m=4$.
As defined in Lemma~\ref{lem:unique_sol}, let $\phi(x)=\mathrm{Perim}\left(P_{8}(x)\right)$. Now, if $x_{0}\in(0,6\pi)$ such that the tangent to the graph of $\phi$ at $\left(x_{0},\phi\left(x_{0}\right)\right)$ passes through $\left(0,0\right)$ (see Figure~\ref{fig:1}), then $x_{0}$ is precisely a solution to the equation $\frac{\phi(x)}{x}=\phi'\left(x\right)$. By Lemma~\ref{lem:unique_sol}, the only solution to this equation is $x_{0}\approx9.34>2\pi$.

\begin{figure}[h] \label{fig:1}
\centering

\includegraphics[scale=0.35]{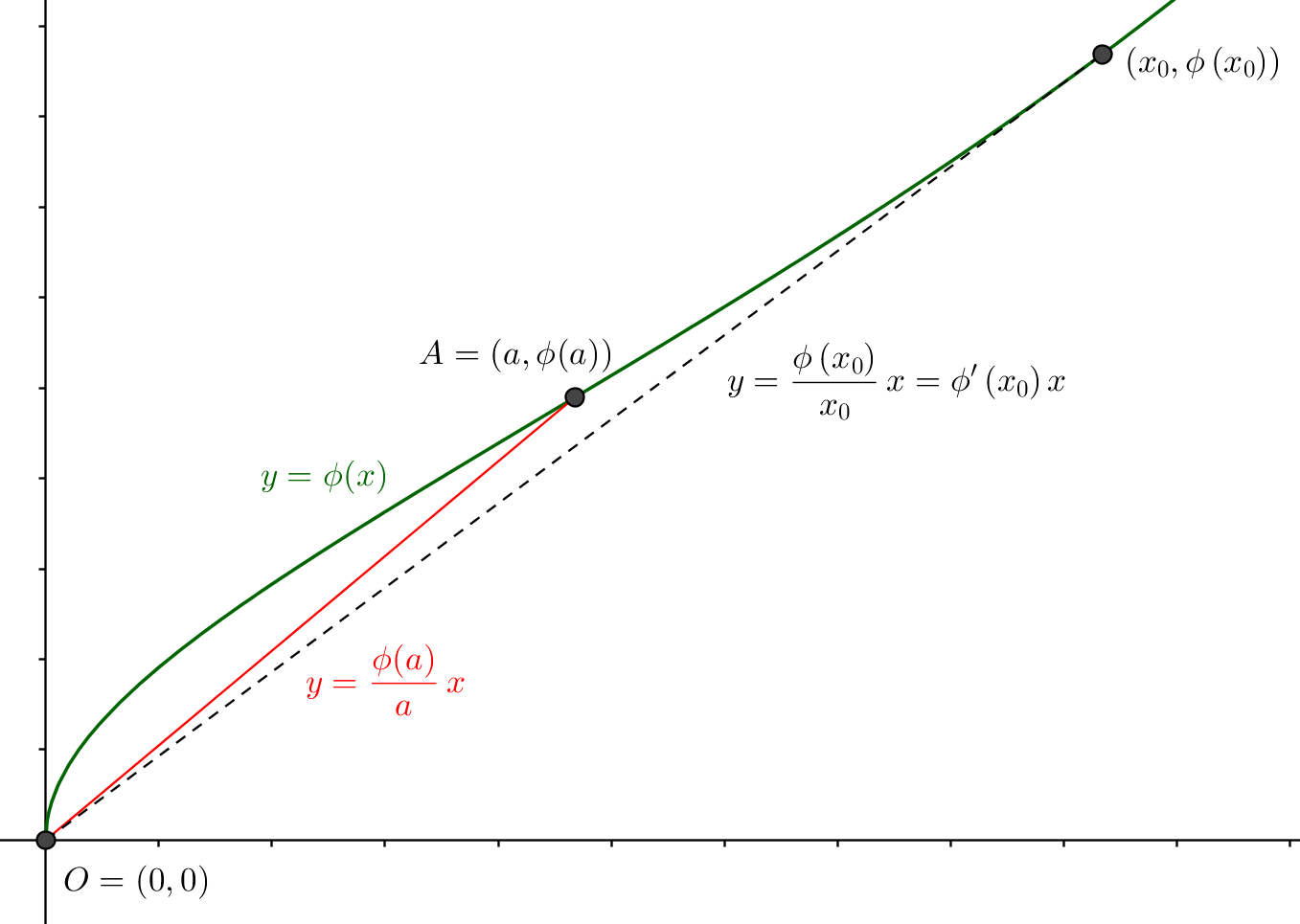}

\caption{The curve $y=\phi(x)$ is above the chord $\overline{OA}$ for $0<a<x_{0}$.}\label{fig:1}

\end{figure}



Fixing $a\in\left(0,x_{0}\right)$, consider the chord $\overline{OA}$
joining the points $O=\left(0,0\right)$ and $A=\left(a,\phi(a)\right)$. We
claim that the graph of $\phi$ does not intersect the chord $\overline{OA}$
in its interior (see Figure~\ref{fig:1}). By Lemma~\ref{lem:unique_sol}, the function $\frac{\phi(x)}{x}$ has a unique critical
point at $x=x_{0}$. We see that $x_{0}$ is in fact a local minimum for
$\frac{\phi(x)}{x}$, since the $\frac{d}{dx} \left(\frac{\phi(x)}{x}\right)<0$ at $x=9$ and $\frac{d}{dx} \left(\frac{\phi(x)}{x}\right)>0$ at $x=10$
(by the computations in the proof of Lemma~\ref{lem:unique_sol}). Thus, the function
$\frac{\phi(x)}{x}$ is strictly decreasing over $\left(0,x_{0}\right)$.
Now, if the graph of $\phi$ intersected the interior of $\overline{OA}$
at some point $Y=\left(y,\phi(y)\right)$,
then we have that $\frac{\phi(a)}{a}=\frac{\phi(y)}{y}$. This is
a contradiction as $a>y$.  Thus, the graph of $\phi$ does not intersect $\overline{OA}$ in its interior.

Since $\lim\limits_{x\to 0} \phi'(x) = +\infty$,  Lemma~\ref{lem:curve_above_chord} gives us that for $a\in\left(0,2\pi\right)\subset \left(0,x_{0}\right)$ and $x\in\left(0,a\right)$, we have $\phi(x)+\phi\left(a-x\right)>\phi(0)+\phi(a)=\phi(a)$. So, 
\begin{align*}
f_{2,4}(x) & =\mathrm{Perim}\left(P_{4}(x)\right)+\mathrm{Perim}\left(P_{8}(a-x)\right)\\
 & \geq\mathrm{Perim}\left(P_{8}(x)\right)+\mathrm{Perim}\left(P_{8}\left(a-x\right)\right)\\
 & \geq\mathrm{Perim}\left(P_{8}(a)\right)=f_{2,4}(0),
\end{align*}
shows that the function $f_{2,4}(x)$ is minimized at $x=0$. 
\end{proof}

\begin{cor} \label{cor:m1=2}
Proposition \ref{Thm1} holds true when $m_{1}=2$.
\end{cor}

\begin{proof}
In Section~\ref{sec:4}, we have proved Proposition~\ref{Thm1} for the cases $n=2$ and $3$. Now, for $n\geq4$, Proposition~\ref{prop:m1=00003D2} implies $$f_{2,n}\left(a_{1}\right)=\mathrm{Perim}\left(P_{4}\left(a_{1}\right)\right)+\mathrm{Perim}\left(P_{2n}\left(a_{2}\right)\right)\geq\mathrm{Perim}\left(P_{2n}(a)\right),$$ which gives us the desired conclusion. 
\end{proof}


\section{The General Result} \label{sec:7}

In this section, we complete the proof of Proposition~\ref{Thm1}. We begin with a lemma which is used in the proof of the proposition. 

\begin{lemma}\label{lem:concave}
For $c>1$, the function $\phi_{c}:[a_{c},\infty)\to\mathbb{R}$, defined by $$ \phi_{c}(x)=2x\cosh^{-1}\left(c \cos\left(\frac{\pi}{2x}\right)\right),$$
 is strictly concave, where  $a_{c}=\frac{\pi}{2\cos^{-1}\left(1/c\right)}$.
\end{lemma}
\begin{proof}
The proof of the lemma uses elementary calculus. We have
\begin{align*}
\phi'_{c}(x)  =&\frac{\pi c\sin\left(\frac{\pi}{2x}\right)}{x\sqrt{c^{2}\cos^{2}\left(\frac{\pi}{2x}\right)-1}}+2\cosh^{-1}\left(c \cos\left(\frac{\pi}{2x}\right)\right) \text{ and }\\
\phi''_{c}(x)  =&-\frac{\pi^{2}c^{2}\sin^{2}\left(\frac{\pi}{2x}\right)}{4x^{3}\left(c \cos\left(\frac{\pi}{2x}\right)-1\right)\sqrt{c^{2}\cos^{2}\left(\frac{\pi}{2x}\right)-1}}
  -\frac{\pi^{2}c^{2}\sin^{2}\left(\frac{\pi}{2x}\right)}{4x^{3}\left(c \cos\left(\frac{\pi}{2x}\right)+1\right)\sqrt{c^{2}\cos^{2}\left(\frac{\pi}{2x}\right)-1}}\\
  &-\frac{\pi^{2}c\cos\left(\frac{\pi}{2x}\right)}{2x^{3}\sqrt{c^{2}\cos^{2}\left(\frac{\pi}{2x}\right)-1}}.
\end{align*}
Thus $\phi''_{c}(x)<0$ for all $x\in \left(a_{c},\infty\right)$ which implies that $\phi_{c}$ is strictly concave.
\end{proof}
\begin{proof}[Proof of Proposition \ref{Thm1}]
This proof draws inspiration from \cite{Gas}. We want to show that $\mathrm{Perim}\left(P_{1}\right)+\mathrm{Perim}\left(P_{2}\right) \geq\mathrm{Perim}\left(P\right)$, or equivalently
\begin{equation}\label{eq:goal}
4m_{1}\cosh^{-1}\left(\frac{\cos\left(\frac{\pi}{2m_{1}}\right)}{\sin\left(\frac{\theta_{1}}{2}\right)}\right)+ 4m_{2}\cosh^{-1}\left(\frac{\cos\left(\frac{\pi}{2m_{2}}\right)}{\sin\left(\frac{\theta_{2}}{2}\right)}\right)  \geq 4n\cosh^{-1}\left(\frac{\cos\left(\frac{\pi}{2n}\right)}{\sin\left(\frac{\theta}{2}\right)}\right)
\end{equation}
by equation~\eqref{eq:perim}, where $\theta_{1},\theta_{2}$ are the interior angles of $P_{1}$ and $P_{2}$ respectively.

According to the notation of Lemma \ref{lem:concave}, let $c_{1}=1/\sin\left(\frac{\theta_{1}}{2}\right)$ and $c_{2}=1/\sin\left(\frac{\theta_{2}}{2}\right)$. Then $a_{{c_1}}=\frac{\pi}{\pi-\theta_{1}}$ and
$a_{c_{2}}=\frac{\pi}{\pi-\theta_{2}}$. In this notation, Equation \ref{eq:goal} is equivalent to $$ 4m_{1}\cosh^{-1}\left(\frac{\cos\left(\frac{\pi}{2m_{1}}\right)}{\sin\left(\frac{\theta_{1}}{2}\right)}\right)+4m_{2}\cosh^{-1}\left(\frac{\cos\left(\frac{\pi}{2m_{2}}\right)}{\sin\left(\frac{\theta_{2}}{2}\right)}\right)=\phi_{c_{1}}\left(m_{1}\right)+\phi_{c_{2}}\left(m_{2}\right).$$
Note that, given $m_{1}, n,\theta_{1}$ and $\theta_{2}$, the equations $m_{1}+m_{2}=n+2$ and $m_{1}\theta_{1}+m_{2}\theta_{2}=n\theta+\pi$ determine $m_{2}$ and $\theta$ uniquely. Furthermore, by the Gauss-Bonnet theorem and $\mathrm{area} \left(P_{1}\right)\geq0$, we have $\theta_{1}\leq\frac{2m_{1}-2}{2m_{1}}\pi$ which implies $m_{1}\geq a_{c_{1}}$. Similarly $m_{2}\geq a_{c_{2}}$.

Given $n, \theta_1$ and $\theta_2$, consider the function $\psi:\left[b_{1},b_{2}\right]\to\mathbb{R}$, defined by $$\psi(x)=\phi_{c_{1}}(x)+\phi_{c_{2}}\left(n+2-x\right),$$ where $b_{1}=\max\left\{ 2,a_{c_{1}}\right\}$ and $b_{2} = \min\left\{ n,n+2-a_{c_{2}}\right\}$. By Lemma~\ref{lem:concave}, the functions $\phi_{c_{1}}(x)$ and $\phi_{c_{2}}\left(n+2-x\right)$ are strictly concave, so $\psi$ is strictly concave. Therefore, the function $\psi$ attains global minimum at the one of the endpoints of its domain.

Now, Corollary~\ref{cor:m1=2} implies that inequality~\eqref{eq:goal} holds true in the cases $b_{1}=2$ and $b_{2}=n$.  If $x = b_1 = a_{c_{1}}>2$, then the equation $m_1 \theta_1 + m_2 \theta_2 = n \theta + \pi$ gives $$ (2n-2) \pi -2n \theta = \left( 2n+2-2a_{c_{1}}\right)\pi-2\left(n+2-a_{c_{1}}\right)\theta_{2}\quad\left(=a\text{, say}\right).$$
Now, inequality~\eqref{eq:goal} equivalent to $$4\left(n+2-a_{c_{1}}\right)\cosh^{-1}\left(\frac{\cos\left(\frac{\pi}{2\left(n+2-a_{c_{1}}\right)}\right)}{\sin\left(\frac{\left(2n+2-2a_{c_{1}}\right)\pi-a}{4\left(n+2-a_{c_{1}}\right)}\right)}\right)\geq 4n \cosh^{-1}\left(\frac{\cos\left(\frac{\pi}{2n}\right)}{\sin\left(\frac{\left(2n-2\right)\pi-a}{4n}\right)}\right),$$ which follows from Lemma \ref{lem:iso_ineq_cont}, as $a_{c_{1}}\leq n$. The case $b_{2}=n+2-a_{c_{2}}$ follows similarly (as $m_{1}$ and $m_{2}$ can be interchanged), so that the inequality \eqref{eq:goal} holds true for $x=b_{1}$ and $x=b_{2}$ in all cases. Thus, we get that
\begin{align*}
\phi_{c_{1}}\left(m_{1}\right)+\phi_{c_{2}}\left(m_{2}\right) & \geq\min\left\{ \phi_{c_{1}}(x)+\phi_{c_{2}}\left(n+2-x\right)\mid x\in\left\{ b_{1},b_{2}\right\} \right\} \\
 & \geq4n\cosh^{-1}\left(\frac{\cos\left(\frac{\pi}{2n}\right)}{\sin\left(\frac{\theta}{2}\right)}\right)
\end{align*}
 as desired.
 \end{proof}

\end{document}